\documentclass[12pt,leqno,twoside]{amsart}

\usepackage{amstext,amsmath,amscd, bezier,indentfirst,amsthm,amsgen,enumerate, geometry,mathrsfs}

\usepackage[all,knot,arc,import,poly,cmtip]{xy}
\usepackage{amsfonts,color, soul}  
\usepackage[colorlinks=true, pdfstartview=FitV, linkcolor=blue, citecolor=blue, urlcolor=blue]{hyperref}
\usepackage{comment}

\usepackage{amssymb}
\usepackage{latexsym}
\usepackage{epsfig}
\usepackage{graphicx,enumerate}
\usepackage{srcltx}
\usepackage{enumitem}

\newtheorem{theorem}{Theorem}[section]
\newtheorem{corollary}[theorem]{Corollary}
\newtheorem{lemma}[theorem]{Lemma}
\newtheorem{proposition}[theorem]{Proposition}

\theoremstyle{definition}
\newtheorem{definition}[theorem]{Definition}
\theoremstyle{remark}
\newtheorem{remark}[theorem]{\sc Remark}
\newtheorem{example}[theorem]{\sc Example}

\renewcommand{\Box}{\square}    

\newcommand{\Sing}{{\mathrm{Sing\hspace{2pt}}}}

\newcommand{\ord}{{\mathrm{ord}}}

\newcommand{\im}{{\mathrm{Im\hspace{1pt}}}}

\newcommand{\ity}{{\infty}}

\newcommand{\e}{\varepsilon}
\newcommand{\m}{\setminus}

\newcommand{\fin}{\hspace*{\fill}$\Box$\vspace*{2mm}}

\newcommand{\cA}{{\mathcal A}}

\newcommand{\cD}{{\mathcal D}}

\newcommand{\cH}{{\mathcal H}}

\newcommand{\cO}{{\mathcal O}}

\newcommand{\cS}{{\mathcal S}}

\newcommand{\cZ}{{\mathcal Z}}

\newcommand{\bR}{{\mathbb R}}
\newcommand{\bC}{{\mathbb C}}
\newcommand{\bK}{{\mathbb K}}

\newcommand{\bN}{{\mathbb N}}
\newcommand{\bZ}{{\mathbb Z}}

\begin{document}

\title[Tame deformations]{Tame deformations of highly singular function germs}

\author{\sc Cezar Joi\c ta}
\address{Institute of Mathematics of the Romanian Academy, P.O. Box 1-764,
 014700 Bucure\c sti, Romania and Laboratoire Europ\' een Associ\'e  CNRS Franco-Roumain Math-Mode}
\email{Cezar.Joita@imar.ro}

\author{\sc Matteo Stockinger}
\address{\' Ecole Normale Sup\' erieure - PSL, 45 rue d'Ulm, 75230, Paris Cedex 05, France}
\email{mstockinger@clipper.ens.psl.eu}

\author{Mihai Tib\u{a}r}
\address{Univ. Lille, CNRS, UMR 8524 -- Laboratoire Paul Painlev\'e, F-59000 Lille, France}
\email{mihai-marius.tibar@univ-lille.fr}





\thanks{The authors acknowledge the partial support of the Labex CEMPI (ANR-11-LABX-0007-01).
C.  Joi\c ta and M. Tib\u ar acknowledges support from the project ``Singularities and Applications'' - CF 132/31.07.2023 funded by the European Union - NextGenerationEU - through Romania's National Recovery and Resilience Plan, and support by the grant CNRS-INSMI-IEA-329. }

\begin{abstract}
We give analytic and algebraic conditions under which a deformation of real analytic functions with non-isolated singular locus is a deformation with fibre constancy. 
\end{abstract}

\maketitle

\section{Introduction}\label{s:intro}

Deformations of function germs is a classical topic with abundant results, and
has a huge impact in geometry, algebra and topology as well as in many more applied fields.
Recently, motivated by solving a technical glitch within a Floer Homology problem,  Ciprian Manolescu \cite{Ma} asked the following: 

\smallskip
\noindent
\textbf{\emph{Question 1}}.  \emph{Consider a family of highly singular real-analytic function germs $F_t : (\bR^n, 0) \to (\bR,0)$,   such that the Jacobian ideal $(\partial F_{t})$ of $F_t$ is independent of $t$ in a small disk at 0. Is the homology of the Milnor fibres of $F_{t}$ ``constant'' in some sense?} 

\smallskip

 In the \emph{complex setting}, for holomorphic function germs with \emph{isolated singularities} (i.e. the very special case of the 0-dimensional singular locus), the constancy of the Jacobian ideal $(\partial F_{t})$ obviously implies that the Milnor number $\mu(F_{t}) = \dim_{\bC}\frac{\cO_{n}}{(\partial F_{t})}$ is constant, and thus the homology of the Milnor fibre is constant too since it is concentrated in the group $\tilde H_{n-1}(F_{t}, \bZ)$, which is free of rank  precisely $\mu(F_{t})$. Moreover, by the L\^{e}-Ramanujam \cite{LR}, Timourian \cite{Tim}, and King \cite{Ki}  results, the constancy of the Milnor number implies that the deformation $F_{t}$ is topologically trivial with the exception of the surface case (i.e. $n=3$) which remains open.

In the case of complex non-isolated singularities, this question seems to be open in general. There is nevertheless a rich literature on equisingularity problems for families of function germs, dealing with various algebraic-geometric sufficient conditions, starting with the Whitney equisingularity studied by Teissier \cite{Te, Te-equi}, extended by Gaffney \cite{Ga}, Houston \cite{Hou} and many others.
 
In the real setting there are results by King \cite{Ki} on the topological triviality of one-parameter families of function germs with \emph{isolated singularity}. While the non-isolated case remains widely open in general,  Parusinski's paper \cite{Pa} goes beyond ``isolated singularity'' in the setting of families of the form $F(x,t) =f(x) +tg(x)$ subject to the condition: 
\begin{equation}\label{eq:parcond}
     | g(x) | \ll  \bigl\| \partial f (x) - t \partial g(x)  \bigr\|   \  \mbox{ whenever } (x,t) \to (x_{0},0) \mbox{,  where } x_{0}\in \{ f=0\}\cap \{g=0\},
\end{equation}
and, remarkably,  shows that such a family is \emph{topologically trivial}, cf \cite{Pa}. 



  
  
  
  \

In this paper we are not interested in the problem of the topological triviality of families of function germs. Instead, we will answer \emph{Question 1} in a different way, by focussing on the constancy of the fibration \eqref{eq:faraway} in a fixed ball, as explained in the following.

\begin{definition}\label{d:topconstantaway}
 Let 
$F:(\bK^n\times\bK, 0)\to(\bK^m, 0)$, $n\ge m\ge 1$, $\bK = \bR$ or $\bC$,  be a $\bK$-analytic map germ,
regarded as a 1-parameter deformation $F_{t}(x) = F(x,t)$ of the function germ $F_{0}= f : (\bK^{n}, 0) \to (\bK, 0)$ having a singular locus of dimension $\ge 1$.

We say that the deformation $F$ of the singular function germ $f$ is \emph{a deformation with fibre constancy}
  if for any small enough radius $r>0$, there exist $0<\delta := \delta(r) \ll r$, and $0<\eta := \eta(r, \delta) \ll \delta$ such that the restriction:
  \begin{equation}\label{eq:faraway}
 (F_{t})_{|} :  B_{r} \cap F_{t}^{-1}(\partial D_{\delta}) \to \partial D_{\delta}.
\end{equation}
is a locally trivial fibration, which is independent,  modulo isotopy, of $r$ and $\delta$,  and of the parameter $t$ with $|t| \le \eta \ll \delta$. 
\end{definition}

In the real setting,  $D_{\delta}$ denotes an interval centred at the origin and its boundary $\partial D_{\delta} = \{ a_{-}, a_{+}\}$ consists of two points. In this setting, the independency asked by Definition \ref{d:topconstantaway} means the invariance of the two fibres\footnote{Of which one may be empty, but  not both empty since we assume that the function germ $F_{0}$ is non-constant.},  $B\cap F_{t}^{-1}(a_{-})$ and $B\cap F_{t}^{-1}(a_{+})$.  
One may then ask:

\smallskip
\noindent
\textbf{\emph{Question 2}}.  \emph{What natural (and minimal) condition implies that the deformation $F$ is a deformation with fibre constancy? }  

\medskip

The problem posed by \emph{Question 2} is not a classical equisingularity problem and makes really sense in the case of $f$ with non-isolated singularities.  If the function germ $F_{0}$ has  an isolated singularity, then for $t\not= 0$ this
may either split into several isolated singularities of the restriction map $(F_{t})_|$ from \eqref{eq:faraway} above, or not split at all.  Even if the isolated singularity of $F_{0}$ splits,  the tube fibrations \eqref{eq:faraway} still exists precisely because the splitting is concentrated at the origin only (i.e. because of the stated choice of a much smaller range of the parameter $t$, namely $|t| \ll \delta$), see e.g. \cite{JiT} for the study of a slightly more general setting of isolated singularities.  
In case it does not split, then one has stronger results: the above cited ones by L\^{e}-Ramanujam \cite{LR} and Timourian \cite{Tim} over $\bC$, and by King \cite{Ki} over $\bR$,  yield the topological triviality of the family of function germs $F_{t}$, or what one calls topological equisingularity. 

In the setting of \emph{non-isolated holomorphic functions}, the concept of   ``deformation with fibre constancy'' (Definition \ref{d:topconstantaway}) appeared under the name of ``admissible  deformations'' in the studies of the topology of certain classes of function germs with non-isolated singular locus,  starting with the seminal paper by Siersma \cite{Si1} on ``line singularities'',  and the follow-up  more general studies of the Milnor fibre of function germs with 1-dimensional and 2-dimension singular loci, see e.g.   \cite{Si-icis},  \cite{Pe}, \cite{dJ}, \cite{Za},  \cite{MS}, \cite{Ne}, \cite{Fe}, \cite{FM}, \cite{ST-mildefo}.   
The very recent paper \cite{Ho} gives algebro-geometric criteria which are sufficient to establish admissibility for complex-analytic deformations.



\smallskip

In both settings, complex and real,  our general answer to \emph{Question 2} is (Theorem \ref{t:jt}): \emph{``tame deformations'' are deformations with fibre constancy}.
 This ``tameness'' condition, introduced in Definition \ref{d:tamedefo}, is based on the transversality of all small enough spheres to the fibres of the deformation $(F_t)_|$ defined at \eqref{eq:faraway} above an appropriate small disk $D_\delta$. This derives from a principle established by Milnor \cite{Mi} and which occurs, mostly under the name of \emph{$\rho$-regularity}, in several studies of fibrations of real analytic map germs on smooth or singular base spaces, such as  \cite{ACT}, \cite{ACT-inf},   \cite{dST0}, \cite{dSTi}, \cite{ART}, \cite{JT}, \cite{CJT}, and see \cite{Ti-book} for older references.

\medskip

Searching for conditions which imply tameness, we focus on two of them, one analytic and the other algebraic.
Let us introduce the first one. It is inspired by Parusinski's condition\footnote{Parusinski's condition \eqref{eq:parcond} had extended a condition used  in the classical case of isolated singularities by Teissier \cite{Te} and L\^{e}-Saito \cite{LS}, and also extended a condition used before for studying   the fibres of  polynomial functions at infinity (\cite{Pa-note}, \cite{ST}, \cite{Ti-compo}, \cite{Ti-book},  \cite[\S 3]{Pa} etc).} \cite{Pa} evoked above as \eqref{eq:parcond}, although much weaker than that.

\begin{definition}\label{d:flat}
Let $F_{t}(x) = F(x,t)$ be an analytic deformation of $f$ (cf Definition \ref{d:topconstantaway}). We consider the following condition: 

\begin{equation}\label{eq:cond}
\left\{
\begin{aligned}
\quad &  \mbox{For any } x_{0}\in \Sing F_{0} \m \{0\}, \mbox{ there is }  c_{x_{0}} >0  \mbox{ such that }  \\
 \quad  &   \biggl| \frac{\partial F}{\partial t} \biggr| \le  c_{x_{0}} \biggl\| \frac{\partial F}{\partial x_{1}}, \ldots , \frac{\partial F}{\partial x_{n}}  \biggr\|   \  \mbox{ when } (x,t) \to (x_{0},0), \mbox{ for } (x,t)\not\in \Sing \widetilde F.
\end{aligned}
\right.
\end{equation}
 where $\widetilde F(x,t) := \bigl(F(x,t), t\bigr)$, where $\widetilde F(x,t)$ is the map germ defined at \eqref{eq:Ftilde}.

\end{definition}

\medskip


 Note that  condition \eqref{eq:cond} excepts\footnote{See  the examples in \S\ref{examples}. Typically such deformations do not satisfy condition \eqref{eq:cond} at the origin $(0,0)$ .} the origin $(0,0)$. This allows the splitting of the singular locus of $F_{0}$ out of the origin (but only at the origin). 
 
 Condition \eqref{eq:cond} may be interpreted as an integral dependence  as follows (see Proposition \ref{p:integral}):  \emph{for any $x_0\neq 0$ close enough to $0$, the function germ 
 $\frac{\partial F}{\partial t}$ belongs to the integral closure 
 $\overline{\bigl( \frac{\partial F}{\partial x_{1}}, \ldots , \frac{\partial F}{\partial x_{n}} \bigr)}$.}
This may also be contrasted to Teissier's condition (c) \cite{Te}, \cite{Te-equi} in the complex setting: \emph{$\frac{\partial F}{\partial t}$ belongs to the integral closure 
 $\overline{\mathfrak{m}_x\bigl( \frac{\partial F}{\partial x_{1}}, \ldots , \frac{\partial F}{\partial x_{n}} \bigr)}$, where
 $\mathfrak{m}_x$  denotes the maximal ideal in $\cO_{n}$}, which
  has been used to study families of holomorphic functions with \emph{isolated singularity}. 
     Still in case of \emph{complex isolated singularities}, Teissier  showed in \cite{Te},  \cite{Te-inv}, that this is also equivalent to the \emph{Whitney equisingularity} of the family of function germs $F_{t}$, which is known to imply its topological triviality.

Here we go beyond the isolated singularity setting, and prove the following:
\begin{theorem}\label{t:genP}
  Let $F(x,t) = F_{t}(x)$ be a $\bK$-analytic deformation of $F_{0}$  which satisfies condition \eqref{eq:cond} of Definition \ref{d:flat}.
  
  Then the deformation $F$ of $F_{0}$ is a deformation with fibre constancy in the sense of Definition \ref{d:topconstantaway}. Moreover, $\widetilde F$ has a Milnor-Hamm tube fibration \eqref{eq:tube}, cf Definition \ref{d:tube}.
\end{theorem}

The proof follows from Theorem \ref{t:jt} after showing in  Theorem \ref{t:tame} that  a nice deformation  is tame (Definition \ref{d:tamedefo}). The tameness follows from  the proof of the existence of the Milnor-Hamm tube fibration which is based on checking the \emph{Thom regularity condition},  a  well-known method employed e.g. in \cite{Hi}, \cite{LS}, \cite{Sa1, Sa2}, \cite{PT}, \cite{dSTi}, \cite{ACT} etc. Here we use a less demanding condition called \emph{$\partial$-Thom regularity} (see \S \ref{ss:partialThom} for details), and that the $\partial$-Thom regularity implies the $\rho$-regularity (Proposition \ref{p:fib}).

\

Our second condition is algebraic. We show by Theorem \ref{t:general}
that if the Jacobian ideal inclusion $(\partial F_{t})\subset (\partial F_{0})$ holds for any $t$ close enough to 0, then condition \eqref{eq:cond} holds, hence
the  deformation $F$ is tame, and thus  it is ``a deformation with fibre constancy'' by our general Theorem \ref{t:tame}. To a certain extent,  this comes closer to Manolescu's \emph{Question 1}. Moreover,  in \S \ref{s:jacocrit} we slightly extend  the setting by replacing  the ideal $(\partial F_{0})$ in the above inclusion by its \emph{integral closure} $\overline{(\partial F_{0})}$.

\

Schematically,  this is what we prove here:\\
\begin{equation}\label{eq:schema}
{\setlength{\fboxrule}{1pt}\fbox{\scriptsize{deformation~with~fibre~constancy }}} 
 \Longleftarrow  \framebox{\scriptsize{tame~deformation}} \Longleftarrow  \framebox{\eqref{eq:cond}} \Longleftarrow  \framebox{\scriptsize{$(\partial F_{t})\subset (\overline{\partial F_{0})}$}}
\end{equation}

\

\

\noindent \textbf{Acknowledgements.} We thank Ciprian Manolescu for bringing to our attention \emph{Question 1}. We also thank him and Lauren\c tiu P\u aunescu for enriching observations  which impelled our study.  Special thanks go to the referee for very interesting remarks which helped us improve the paper.

C. Joi\c ta and M. Tib\u ar acknowledge support of the project ``Singularities and Applications'' -  CF 132/31.07.2023 funded by the European Union - NextGenerationEU - through Romania's National Recovery and Resilience Plan,  and support by the grant CNRS-INSMI-IEA-329.

\section{Image and discriminant of a deformation}\label{s:image}

Let $F_{0}:(\bK^{n},0) \to (\bK, 0)$ be a non-constant analytic function germ.
Let  $F_{t}(x):= F(x, t)$ be an analytic deformation of the analytic function germ $F_{0}(x) := F(x,0)$ such that $F(0,t) =0$ for any $t$ in some small neighbourhood of $0\in \bK$.
Let 
\begin{equation}\label{eq:Ftilde}
  \widetilde F := (F(x,t),t) : (\bK^{n}\times \bK, 0) \to (\bK \times \bK, 0)
\end{equation}
 be the analytic   map germ defined by  the deformation $F$.
The singular set   $\Sing \widetilde F$ is the zero locus $Z\bigl( \frac{\partial F}{\partial x_{1}}(x,t), \ldots , \frac{\partial F}{\partial x_{n}}(x,t)\bigr)$. As a set germ at the origin $(0,0)$, it contains the union of set germs $\bigcup_{t\in \bK}\Sing F_{t}$, but may contain other irreducible components, see e.g. Example \ref{ex:notJ}. Note that   $\Sing \widetilde F \cap \{ t=0\} = \Sing F_{0}$, and that we have the inclusion $\Sing F\subset \Sing \widetilde F$, which may be a strict inclusion even when restricted to the slice $\{ t=0\}$, like in the example: $F(x,y,t) = x^{2} +ty$.

We will consider here the class of  deformations $F$ with the property  that the inclusion $\Sing F\subset \Sing \widetilde F$ restricts to an equality in the slice  $\{ t=0\}$, namely deformations $F$ such that:
\begin{equation}\label{eq:cond0}
    (\Sing F)_{| t=0} = (\Sing \widetilde F)_{| t=0}.
 \end{equation}  
We refer to Proposition \ref{p:delta2} for the relations of \eqref{eq:cond0} with other conditions.  See also the Examples section \S\ref{examples}
for more comments.

 \subsection{The image problem}\label{ss:imag}
 Let us start by observing that analytic function germs have well-defined images as set germs, either over $\bC$ or over $\bR$. This is the case for our function germs $F_{t}$, for any $t$. 
 However this is no more the case for the image of a map germ $G: (\bK^{m}, 0) \to (\bK^{2}, 0)$  with $m\ge 2$, as pointed out in \cite{JT}, see also \cite{ART}. The map $G$ may be locally open, i.e. one may have the equality of set germs $(\im G, 0) = (\bK^{2},0)$, but certain map germs, for instance $G(x,y) = (x, xy)$, do not even have well-defined image as a set germ at the origin.  For the problem ``when a map germ has a well defined image as a set germ'' we refer to \cite{JT-arkiv}, see also \cite{ART} and \cite{JT}.
 
 In our setting, the map germ $\widetilde F$ associated to an analytic deformation $F(x,t)$ is somewhat more special.
\begin{proposition}\label{p:image}
In case $\bK = \bC$, the map germ $\widetilde F$ is locally open. \\
In case $\bK = \bR$ we have:
\begin{enumerate}
\rm  \item \it   If  $F_{0}$ is  locally open then $\widetilde F$ is a locally open map germ.
In particular this is the case if $\Sing F_{0} \not= F_{0}^{-1}(0)$.
\rm  \item \it   If  $ F_{0}$ is  not locally open then for any radius $r>0$,  the image $\widetilde F(B_{r})$  contains  
either the set germ $(\bR_{\ge 0} \times \bR, 0)$ or the set germ $(\bR_{\le 0} \times \bR, 0)$.
\end{enumerate}
\end{proposition}
\begin{remark} 
In the above point (b) one may not have equality, therefore the word ``contains'' is important. Unlike the complex setting, in the real setting the map $\widetilde F$ might not have a well-defined image as a set germ at $0\in \bR^{2}$. This may happen in Proposition \ref{p:image}(b), as shown by the following very simple example:
 $F(x,y, t)=x^2+t(x+y)$ is a deformation of $F_{0}:(\bR^2,0) \to (\bR,0)$, $F_{0}(x,y)=x^2$.
 
 Note that if we view this example over $\bC$ then the image of $\widetilde F$ is a well-defined set germ, by Proposition \ref{p:image}.
\end{remark}
\begin{proof}[Proof of Proposition \ref{p:image}]\ \\
\noindent  \underline{Case $\bK = \bC$.}  The zero set $\widetilde F^{-1}(0,0)= F_{t}^{-1}(0)$ has codimension 2 in the source $\bC^{n}\times \bC$ of the map $\widetilde F$. Therefore   \cite[Theorem 1.1(a)(i)]{JT-arkiv} applies here, and tells that $\widetilde F$ is locally open.

\smallskip
\noindent   \underline{Case $\bK = \bR$.} 
\noindent Let $B_{r}\subset \bR^{n}\times \bR$ denote the open ball of radius $r>0$ centred at the origin, and let $B'_{r'}\subset \bR^{n}$ such a ball in $\bR^{n}$. We need the following result:

\begin{lemma}\label{l:elem}
  If $\im F_0\supset  \bR_{\ge 0}$ as set germs,
then for any $r>0$, we have $\widetilde F(B_{r})  \supset  \bR_{\ge 0} \times \bR$.\\
 The similar statement holds if we replace $\bR_{\ge 0}$ by $\bR_{\le 0}$.
\end{lemma}
\begin{proof}
The proof is elementary and uses only the continuity of the maps;  we provide it for completeness.

Our deformation of $F_{0}$ presents as $F(x,t)=F_0(x)+tG(x,t)$ with $G$ analytic, hence continuous, where $G(0,t)=0$ for all $t$ close enough to 0. Let  us fix some small enough radius $r>0$, 
such that $|G(x,t)|\leq 1$ on $B_{r}$. Let $r'>0$ and $\eta'>0$ such that $B'_{r'}\times (-\eta' , \eta' ) \subset B_{r}$.

 By our hypothesis $\im F_0\supset(\bR_{\ge 0}, 0)$, there exists $\e_0>0$ such that $F_0(B'_{r'})\supset[0,\e_0)$.
 Setting $\e:=\frac{\e_0}4$ and $\eta :=\min\{\eta' , \e\}$, we will show  that 
$\widetilde F(B_r)\supset [0,\e)\times(-\eta,\eta)$, as follows. 

Let $(\alpha,\beta)\in [0,\e)\times(-\eta,\eta)$. 
We have to show that there exits $x\in B'_{r'}$ such that $F(x,\beta)=\alpha$. Let $x_1\in B'_{r'}$ such that $F_0(x_1)=\frac{\e_0}2$.
Since $|G(x_1,\beta)|\leq 1$ and  $|\beta|<\e \le \frac{\e_0}4$, we get:
$$F(x_1,\beta)=\frac{\e_0}2+\beta G(x_1,\beta)>\frac{\e_0}4 = \e >\alpha. $$

Since we also have $F(0,\beta)=0\leq\alpha$, and since $F$ is continuous,
we conclude that there exists a point $x$ on the segment joining the origin with $x_1$, hence in $B'_{r'}$, such that $F(x,\beta)=\alpha$. Our claim is proved.

Mutatis mutandis, the same proof applies if we replace $\bR_{\ge 0}$ by $\bR_{\le 0}$.
\end{proof}

\noindent (b).  If $F_{0}$ is not locally open, then its image is either the set germ $(\bR_{\ge 0},0)$
or the set germ $(\bR_{\le 0},0)$.   We may thus apply Lemma \ref{l:elem} to conclude.

\medskip
\noindent (a). The first statement of  (a) is also a direct consequence of  Lemma \ref{l:elem}.

Let us prove the second statement of (a). We have $\Sing F_{0} =\Sing \widetilde F \cap \widetilde F^{-1}(0,0)$ and by our hypothesis this is included but not equal to $F_{0}^{-1}(0) = \widetilde F^{-1}(0,0)$ which is the central fibre of the map $\widetilde F$. We may therefore apply  \cite[Lemma 2.5]{JT} to the map $\widetilde F$ and  conclude that it is a locally open map.
\end{proof}

 \subsection{The discriminant of a deformation}\label{ss:discrim}
The image by $G$ of the singular locus $\Sing G$ of a map germ $G: (\bK^{m}, 0) \to (\bK^{s}, 0)$, $m\ge s\ge 2$,  supports the same discussion. It has been pointed out in \cite{JT} that this might not be well-defined as a set germ, and moreover,  this may happen even if $\im G$ is a well defined set germ.  

Nevertheless, the case of a deformation $F$ turns out to be more special; the next result tells that  the image  $\widetilde F (\Sing \widetilde F)$ is a well-defined set germ.  

\begin{proposition}\label{p:delta1}
If the dimension  of the target of the non-constant analytic map germ $G: (\bK^{m}, 0) \to (\bK^{s}, 0)$ is $s=2$,  then $G(\Sing G)$ is a well-defined set germ. It is either empty, or one point (the origin), or it is as follows: in the case $\bK = \bC$ it is a complex analytic curve, whereas in case $\bK = \bR$ it is a semi-analytic curve.

In particular, the image  $\Delta := \widetilde F (\Sing \widetilde F) $ is a well-defined set germ, that we will call
``discriminant of the deformation $F$''.    
\end{proposition}
\begin{proof}
\noindent By \cite[Theorem 3.2]{JT},  for any non-constant holomorphic map germ $G$ of target $\bC^{2}$, the image $G(\Sing G)$  is a well-defined complex analytic set germ. Since the complement of this image is dense (by Sard's Theorem),  the image cannot have dimension 2.

In the setting of real deformations we do as follows: for any fixed ball $B_{r}$, the image $G (B_{r}\cap \Sing G)$ is a subanalytic subset of $\bR^{2}$, which is in fact semi-analytic as proved by \L ojasiewicz \cite{Lo} when the target is of  dimension 2. This set is included in the complex discriminant of $G$ viewed as a holomorphic function, which is a well-defined germ of a complex curve, as shown above. It then follows that 
$G(B_{r}\cap \Sing G)$ is a real curve and that it does not depend on the radius $r$ regarded
as a germs at the origin of $\bR^{2}$.  

This proof applies to $G := \widetilde F$,  which shows our second claim. 
\end{proof}

 \section{Tame deformations}\label{s:tamedefo}
 
 In the preceding section we have seen that the real setting is different from the complex one in what concerns the image of the map $\widetilde F$. However we are interested here in certain general fibres of the deformation, and therefore we will
 adapt the study to this more particular interest.
 
 \smallskip
 
\noindent\emph{\underline{Notation}} $\cH_{F}$.  
In the complex setting, $\cH_{F}$ will denote the full target $\bC^{2}$. 
In the real setting, according to the 3 situations in Proposition \ref{p:image}(a) and (b), we have: 
\begin{enumerate}
  \item $\cH_{F} := \bR^{2}$  if $\widetilde F$ is locally open.
  \item $\cH_{F}:= (\bR_{\ge 0} \times \bR, 0)$, or $\cH_{F}:=  (\bR_{\le 0} \times \bR, 0)$,  if $\widetilde F$ is not locally open but $\widetilde F(B_{r})$ contains this half-plane for any $r>0$, respectively.
\end{enumerate}
  
 \subsection{Local tube fibration}\label{ss:tube}
  We have seen that the discriminant $\Delta := \widetilde F(\Sing \widetilde F)$ is a well-defined set germ, cf Proposition
 \ref{p:delta1}.
  \begin{definition}\label{d:tube}
 We say that \emph{$\widetilde F$ has a local tube fibration},  also called \emph{Milnor-Hamm fibration},  if for any small enough $r >0$ there exists a radius $\delta\ll r$,  and a radius $\eta \ll \delta$ such that the restriction:
\begin{equation}\label{eq:tube}
  \widetilde F_{|} : \overline{B_{r}} \cap  \widetilde F^{-1}(\cH_{F}\cap (D_{\delta}\times D_{\eta})\m \Delta)   \to \cH_{F}\cap (D_{\delta}\times D_{\eta})\m \Delta
\end{equation}
is a locally trivial fibration which is independent of the chosen constants up to isotopy. 
\end{definition}

 The above definition sounds a bit different from the general definition of the Milnor-Hamm fibration in \cite{ART} since here we do not assume that the map germ $\widetilde F$ is ``nice'', i.e.,  both $\im \widetilde F$ and $\widetilde F (\Sing \widetilde F)$ to be well-defined as set germs at the origin.  Let us explain what happens in our particular situation.
  
 In the  holomorphic setting,  the map $\widetilde F$ is  automatically nice,  by Proposition \ref{p:image}.
  If the local tube fibration \eqref{eq:tube} exists then  the fibre $B_{r}\cap F_{t}^{-1}(s)$ is diffeomorphic to the fibre  $B_{r}\cap F_{0}^{-1}(s)$, for any $(s,t) \in (D_{\delta}\times D_{\eta})\m \Delta$ because this later set is connected. This means that one has a single fibre of \eqref{eq:tube}.
  
 
   In the real analytic setting, $\cH_{F}$ is included in the target of $\widetilde F$ by Proposition \ref{p:image}, and therefore the target of the map $\widetilde F_{|}$ defined by \eqref{eq:tube} is independent of the radius of the ball $B_{r}$.  As $\Delta$ is well-defined, it follows that $\im \widetilde F_{|}$ is well-defined as a set germ, and therefore $\widetilde F_{|}$ is a ``nice map'' in the sense of \cite{ART}. The set  $\cH_{F}\cap (D_{\delta}\times D_{\eta})\m \Delta$ consists of finitely many connected components $A_{1}, \ldots, A_{\zeta}$, each of them having a unique fibre $B_{r}\cap F_{t}^{-1}(s)$  for  $(s,t) \in A_{j}\subset \cH_{F}\cap  (D_{\delta}\times D_{\eta})\m \Delta$, up to diffeomorphisms.  
   
   In particular we have:
 
\begin{corollary}\label{c:trivbd}
 If the local tube fibration \eqref{eq:tube} exists, then  the deformation  $F(x,t)$ is a deformation with fibre constancy  (in the sense of Definition \ref{d:topconstantaway}).
\end{corollary}
\begin{proof}
 Let $B'_{r}\subset \bK^{n}$ be a small enough ball centred at the origin of radius $r>0$ such that it is a Milnor ball for the function germ $F_{0}$, and let $0<\delta\ll r$ be a corresponding Milnor disk. 
  As above, let $B_{r}\subset \bK^{n}\times \bK$ be the ball centred at the origin of the same radius $r$.
 
 Since $\Sing \widetilde F$ intersects the slice $\{t=0\}$ at $\Sing F_{0}\subset F^{-1}(0)$, it follows that   for  small enough $0<\eta \ll \delta$ we have:
 $$\Delta \cap \cH_{F}\cap  (\partial D_{\delta}\times  D_{\eta}) = \emptyset.$$
 
This shows that \eqref{eq:faraway} is a sub-fibration of \eqref{eq:tube} for any $t\in D_{\eta}$.
 \end{proof}
 
\begin{remark}
 The fibre  $B_{r}\cap  F_{0}^{-1}(s)$ is precisely a Milnor fibre of $F_{0}$, since we have assumed  that the ball $B_{r}$ is a Milnor ball for $F_{0}$, and it is diffeomorphic to  $B_{r}\cap F_{t}^{-1}(s)$.
Nevertheless
 $B_{r}\cap F_{t}^{-1}(s)$ is not necessarily the Milnor fibre of the function germ $F_{t}$. Indeed, the function germ $F_{t}$
 may require the choice of a smaller radius $r'<r$ such that $B_{r'}$ is Milnor ball for $F_{t}$ at 0.  This problem is well-known in deformations (see e.g. \cite{JiT}): for each $t\in D_{\eta}$ there is a ``maximum radius'' $r_{t}$ of the Milnor ball of the function germ $F_{t}$, and it is possible that $\lim_{t\to 0}r_{t} =0$, whereas $r_{0}:= r$ is a well-defined positive value.
\end{remark}

  
\smallskip

\begin{definition}(Milnor set)\label{d:milnorset}\ \\
Consider the square of the Euclidean distance function from the origin $\rho : \bK^{n+1} \to \bR_{\ge 0}$, and the map
$$(\widetilde F, \rho)_{|} : \widetilde F^{-1}((D_{\delta}\times D_{\eta})\m \Delta) \to \bK^{2} \times \bR_{\ge 0}
$$ 
The analytic set
$$M(\widetilde  F) := \overline{\Sing (\widetilde F, \rho)_{|} \setminus \widetilde F^{-1}(\Delta)}  $$ 
will be called here the \emph{Milnor set} of $\widetilde F$. 
\end{definition}

In the following we will tacitly consider the germ of $M(\widetilde F)$ at $(0,0)$, for which we use  the same notation.
 
 %
\begin{definition}(Tame deformations)\label{d:tamedefo}\ \\
 We say that the deformation $F_{t}(x) = F(x,t)$ of $F_{0}$ is \emph{tame} if and only if the following condition holds:
\begin{equation}\label{eq:rhoreg}
 M(\widetilde F)\cap \{t=0\} \cap \Sing F_{0} \subset \{(0,0)\}.
\end{equation}
\end{definition}
\begin{remark}\label{r:tame-proper}
 Condition \eqref{eq:rhoreg} is equivalent to the \emph{$\rho$-regularity} of $\widetilde F$ considered in \cite[(6)]{ART}, despite the fact that the definition of the Milnor set $M(\widetilde F)$ itself is slightly different from the one considered in \cite{ART}.  
\end{remark}

\begin{theorem}\label{t:jt}
 If the analytic deformation $F(x,t)$  of $F(x,0)$ is tame, then  $\widetilde F$ has a Milnor-Hamm tube fibration \eqref{eq:tube}, and the deformation  $F(x,t)$ is a deformation with fibre constancy  (in the sense of Definition \ref{d:topconstantaway}).
\end{theorem}


\begin{proof}
It was proved in  \cite[Lemma 3.3]{ART}: \emph{Let $G: (\bR^{m}, 0) \to (\bR^{p}, 0)$, $m \ge p > 1$, be a non-constant nice analytic map germ. If $G$ is $\rho$-regular then G has a Milnor-Hamm tube fibration.}  

We have seen in \S \ref{ss:imag} that in case of a deformation $F$, the condition ``nice'' (recalled after Definition \ref{d:tube})  holds in the holomorphic setting,
 thus the above result applies and yields our claim.

In the real setting, we have shown in the comments after Definition \ref{d:tube} that the condition ``nice'' holds for the map $\widetilde F_{|}$ in \eqref{eq:tube}.
The arguments of the proof  of  \cite[Lemma 3.3]{ART} work without  any modification.  Indeed, the $\rho$-regularity \eqref{eq:rhoreg} implies that  the map \eqref{eq:tube} is a proper submersion,  thus it is a locally trivial fibration over $\cH_{F}\cap (D_{\delta}\times D_{\eta})\m \Delta$,
 and  the independency of the tube fibration \eqref{eq:tube} follows from the definition.

Finally we may apply Corollary \ref{c:trivbd}, and our claim is proved in both  complex or real settings.

Let us note for the record that the tameness of the deformation $F$ is equivalent to  the fact that \eqref{eq:tube} is a proper stratified submersion independently of the chosen constants.
\end{proof}

 Although $M(\widetilde F)$ is an analytic set,  hence defined by equations, it is still not easy to compute it since one has to single out only those irreducible components  which are not included in $\widetilde F^{-1}(\Delta)$.  We  will display in the following two conditions which imply ``tameness'', one analytic and the other algebraic.

\section{A condition for the tameness}\label{s:parucrit}

\subsection{The partial Thom stratification}\label{ss:partialThom}

 Let us recall the \emph{$\partial$-Thom regularity}, after \cite[Def. 2.1]{Ti-compo}, \cite[A 1.1]{Ti-book}, \cite[\S 6]{DRT},  and  \cite[\S 4]{ART}, see also \cite{CJT}.
This is a  weaker condition than the usual Thom regularity, but sufficient for the existence of  the Milnor tube fibration (cf Definition \ref{d:tube}) in the general case of real map germs, see Proposition \ref{p:fib} below.  There are however examples where the map germ $G$ has Milnor tube fibration without being  $\partial$-Thom regular, see \cite{PT} and \cite{ART}. 

Let $G:(\bR^m, 0) \to (\bR^p,0)$ be a non-constant  analytic map germ, and let $\Delta$ denote its discriminant. 
For a given stratification of a neighbourhood of $0 \in \bR^{m}$, let  $A, B$ be strata such that $B\subset \bar A\m A$. One says that the pair $(A,B)$ satisfies the Thom (a$_G$)-regularity condition at $x\in B$,  if the following condition holds: for any $\{x_k\}_{k\in \bN}\subset A$ such that  $x_k\to x$, if the tangent space
$T_{x_k}(G_{|A})$ converges, when $k\to \ity$,  to a limit $H$ in the appropriate Grassmannian,  then $T_{x}(G_{|B})\subset H$. 

In the case of our function germ $F$, it is known that there exists a \emph{Thom (a$_{F}$)-regular stratification}  of  $F^{-1}(0)$, cf. Hironaka \cite{Hi}: any real or complex function germ $h$ admits a Thom (a$_{h}$)-regular stratification of its zero locus $h^{-1}(0)$; see also the discussion in \cite{ART}.
Moreover, in the complex analytic setting, a Whitney (b)-regular stratification of $h^{-1}(0)$ is also Thom (a$_{h}$)-regular. This was shown in \cite[Thm. 4.2.1]{BMM} with $\cD$-module techniques;  topological proofs can be found in  \cite{Pa-note} and \cite{Ti-compo}, see also \cite[Thm. A 1.1.7]{Ti-book}.

\begin{definition}\label{d:partialthom} (after \cite{ART})\ \\
We say that $G$ is \textit{$\partial$-Thom regular} if there is a ball $B^{m}_{r}$ centred at $0\in \bR^{m}$ and a semi-analytic stratification $\cS=\{S_{\alpha}\}$ of $B^{m}_{r}\cap G^{-1}(0)$ such that, for any stratum $S_\alpha$, the pair $(B^{m}_{r}\m G^{-1}(\Delta), S_\alpha)$  satisfies the Thom (a$_G$)-regularity condition. 

We also say that the stratification $\cS$ is a \emph{$\partial$-Thom} (a$_{G})$-\emph{stratification}.
\end{definition}
Let us also point out that any Thom regular stratification is obviously $\partial$-Thom regular.
 The difference is that we do not ask neither the Whitney (a)-regularity  condition between couples of strata $A,B$ inside $G^{-1}(0)$, nor the (a$_G$)-regularity for $A$ outside $G^{-1}(0)$ and $B$ inside $G^{-1}(0)$.

 The $\partial$-Thom regularity is however sufficient to insure the existence of the Milnor-Hamm fibration:
\begin{proposition}\label{p:fib}\cite[Prop. 4.2]{ART} and \cite[Cor. 5.8]{JT}. \label{th:tube}
 If $G$ is $\partial$-Thom regular,  then $G$ is $\rho$-regular and has a Milnor-Hamm fibration. \fin
\end{proposition} 

\subsection{Theorem and its proof}\label{ss:thm}

Our $\bK$-analytic deformation  of $F_{0}$ has by definition the following presentation:
 \begin{equation}\label{eq:present}
 F(x,t) =f_0(x)+\sum_{j\geq 1}^\infty t^jf_j(x)
 \end{equation}
 where, for any $j\ge 0$, $f_j$ is a $\bK$-analytic function germ at 0 of variable $x\in \bK^{n}$, with $f_j(0)=0$, since we have assumed that  $F_{t}$ is a function germ at the origin, and thus $F_{t}(0) = 0$, for any $t$ close enough to 0.

\begin{theorem}\label{t:tame}
Let $F_{t}(x) =F(x,t)$ be an analytic deformation of $F_{0}$ which satisfies the condition \eqref{eq:cond}  of Definition \ref{d:flat}.
 Then the deformation $F(x,t)$ of $F_{0}(x)$ is $\partial$-Thom regular,  and therefore tame.  
  \end{theorem}
\begin{proof}
Let  us fix a semi-analytic Whitney stratification $\cS$ of  $\bK^{n}\times \bK$ which is Thom (a$_{F}$)-regular, thus $\Sing F$ is a union of strata. Let  $(y,0)\in \bK^{n}\times \bK$, with $y\not= 0$,  be a point on some stratum $V\in \cS$, $V\subset \Sing F \subset  F^{-1}(0)$.  That $(y,0)\in \Sing F$  is equivalent to $y\in \Sing F_{0}\cap \bigl\{\frac{\partial F}{\partial t}(x,0) =0\bigr\}\m \{0\}$.  Moreover, this is also equivalent to  $y\in \Sing F_{0}\m \{0\}$. Indeed, the inequality \eqref{eq:cond} applied at $y\in \Sing F_{0}\m \{0\}$ shows that $y\in  \bigl\{\frac{\partial F}{\partial t}(x,0) =0\bigr\}$, which proves the inclusion: 
\begin{equation}\label{eq:inclsing}
 \Sing F_{0}\m \{0\}  \subset \biggl\{\frac{\partial F}{\partial t}(x,0) =0\biggr\}\m \{0\}.
\end{equation}
Let us remark that this proof also shows that condition  \eqref{eq:cond} implies condition \eqref{eq:cond0}. 

\smallskip

Let now $T_{(x_{t},t)}F^{-1}(s_{t})$ denote the tangent space at some smooth point $(x_{t},t)$ of the fibre, i.e. $(x_{t}, t)\not\in \Sing F$ and  $s_{t}:=F(x_{t},t)$. 
 The assumed $\partial$-Thom (a$_{F}$)-regularity condition at $(y,0)$ amounts to the following property: 
 for any choice of a sequence $(x_{t},t)\to (y,0)$ such that $(x_{t},t)\not\in \Sing F$, we have the inclusion:
 
\begin{equation}\label{eq:incl1}
  \lim_{(x_{t},t)\to (y,0)}T_{(x_{t},t)}F^{-1}(s_{t}) \supset T_{(y,0)}V,
\end{equation}
where we may assume without loss of generality that the limit exists in the appropriate Grassmannian.

We  now consider the slice of the stratification $\cS$ by $t=0$, consisting of the sets $V':=V\cap \{t=0\}$ for all $V\in \cS$. There exists the roughest semi-analytic Whitney (a)-regular stratification $\cS'$ of the central fibre $\widetilde F^{-1}(0,0) = F_{0}^{-1}(0)$ which refines this slice stratification, thus the sets $V'$ are unions of strata of $\cS'$.

\begin{lemma}
 If   condition \eqref{eq:cond} holds, then $\cS'$ is a $\partial$-Thom stratification for the map $\widetilde F$. 
\end{lemma}


\begin{proof}
We consider the fibres of $F_{t}$ for all $t$ close enough to 0, 
and make them converge to the central fibre $F_{0}^{-1}(0)$. 
We need to prove the $\partial$-Thom (a$_{\widetilde F}$)-regularity condition at the point $(y,0)$, which amounts to showing that
for any choice of a sequence $(x_{t},t)\to (y,0)$ such that $(x_{t},t)\not\in \Sing F$ and $s_{t}:= F_{t}(x_{t})$, we have the inclusion:
\begin{equation}\label{eq:incl2}
  \lim_{(x_{t},t)\to (y,0)}T_{(x_{t},t)}(F_{t}^{-1}(s_{t})\times\{t\}) \supset T_{(y,0)}V',
\end{equation}
where by $V'$ we denote the stratum of $\cS'$ which contains $(y,0)$, in particular we have $T_{(y,0)}V\supset  T_{(y,0)}V'$.

We will deduce \eqref{eq:incl2} from the inclusion \eqref{eq:incl1} via the condition \eqref{eq:cond}.
Let us suppose by \emph{reductio ad absurdum} that we have:
\begin{equation}\label{eq:nonincl}
\lim_{(x_{t},t)\to (y,0)}T_{(x_{t},t)}(F_{t}^{-1}(s_{t})\times\{t\})\not \supset T_{(y,0)}V'.
\end{equation}
Both sides are vector subspaces 
of $\bR^n\times\{0\}$, the left hand side has dimension  $n-1$, and the right hand side has dimension $\dim T_{(y,0)}V'\geq 1$.
Our assumption \eqref{eq:nonincl} implies the equality: 
\begin{equation}\label{eq:incl3}
 T_{(y,0)}V'+ \lim_{(x_{t},t)\to (y,0)}T_{(x_{t},t)}(F_{t}^{-1}(s_{t})\times\{t\})=\bR^n\times\{0\}.
\end{equation}

On the other hand, both sides of  \eqref{eq:nonincl} are included in 
$\lim_{(x_{t},t)\to (y,0)}T_{(x_{t},t)}F^{-1}(s_{t})$ and the later has dimension $n$. 

By dimension reasons, we therefore get from \eqref{eq:incl3} the equality:
$$\lim_{(x_{t},t)\to (y,0)}T_{(x_{t},t)}F^{-1}(s_{t})=\bR^n\times\{0\}.$$
This equality is equivalent to:
\begin{equation}\label{eq:vector}
\lim_{(x_{t},t)\to (y,0)}\frac{1}{\| \frac{\partial F}{\partial x_1},\dots,\frac{\partial F}{\partial x_n},\frac{\partial F}{\partial t}\|}
\biggl(\frac{\partial F}{\partial x_1},\dots,\frac{\partial F}{\partial x_n},\frac{\partial F}{\partial t}\biggr)=(0,\dots,0,\pm 1),
\end{equation}
in particular we have $\lim_{(x_{t},t)\to (y,0)} \frac{\bigl|\frac{\partial F}{\partial t}\bigr|}{\|\frac{\partial F}{\partial x_1},\dots,\frac{\partial F}{\partial x_n},\frac{\partial F}{\partial t}\|} =1$.

The inequality   \eqref{eq:cond} tells now that if we divide it by $\bigl\|(\frac{\partial F}{\partial x_1},\dots,\frac{\partial F}{\partial x_n},\frac{\partial F}{\partial t})\bigr\|$ then, by taking the limit,  we get the inequality:

  \[  \lim_{(x_{t},t)\to (y,0)} \biggl\| \frac{\partial F}{\partial x_1},\dots,\frac{\partial F}{\partial x_n}\biggr\|
  / \biggl\| \frac{\partial F}{\partial x_1},\dots,\frac{\partial F}{\partial x_n},\frac{\partial F}{\partial t}\biggr\|
     \ge \frac{1}{c_{y}}>0
  \]
This  tells that the first $n$ entries in  \eqref{eq:vector} cannot be all zero. We have thus shown the inclusion \eqref{eq:incl2}, which ends our proof  that the Thom (a$_{\widetilde F}$)-regularity holds at the point $(y,0)$.
\end{proof}


 The existence of the partial Thom stratification $\cS'$ for the map $\widetilde F$ implies that there is  $R>0$ such that,  for any positive $r\le R$, the sphere $S_{r}\subset \bK^{n}$ is transversal to all  positive dimensional strata of $\cS'$. As a direct consequence of the definition \eqref{eq:incl1}, it follows that
   the sphere $S_{r}$ is transversal to the smooth nearby fibres of $\widetilde F$ as in \eqref{eq:tube}, and therefore $F$ is tame (and in particular, by Theorem \ref{t:jt},  the tube fibration \eqref{eq:tube} exits).
 This ends the proof of Theorem \ref{t:tame}.
\end{proof}



\section{The Jacobian criterium}\label{s:jacocrit}


Before introducing the Jacobian criterion for tameness, let us first recall the \emph{integral closure} of an ideal of germs of analytic functions. We denote by $\cA_n$ the ring of germs of analytic functions at the origin on $\bK^n$. 

\begin{definition}[\cite{Te}, \cite{Ga}]\label{d:integral}
The  integral closure of  an ideal $J\subset \cA_n$, denoted by $\overline J$, is the set of all  $f\in \cA_n$
such that for any analytic arc $\mu:(\bK,0)\to(\bK^n,0)$ one has $\mu^{*}f \in(\mu^{*}J)\cA_1$, equivalently:
$$\ord_{s}f(\mu(s))\geq\min\{\ord_{s}h(\mu(s)) \mid h\in J\}.$$
\end{definition}

\begin{remark}\label{r:integral}
The above definition was introduced by Gaffney \cite{Ga}  in the real analytic setting.
In the complex analytic setting, it was proved by Teissier \cite{Te}
that the usual algebraic definition of the integral closure is equivalent to the above one. 
\end{remark}

\begin{proposition}\label{p:integral}{\cite[1.3.1, Proposition 1]{Te-LN},  \cite[Proposition 4.2]{Ga}}\ \\
 Let $J$ be an ideal of $\cA_n$. Then 
$f\in\overline J$ if and only if, for every set of generators $\{g_j\}$ of $J$, there exists a neighbourhood 
$U$ of $0\in\bK^n$, and a constant $c >0$, such that $|f(x)|\leq  c\cdot \sup_j|g_j(x)|$ for all $x\in U$.
\fin
\end{proposition}

\begin{corollary}\label{c:integral} 
Let $f\in J$,  let $\{g_j\}$ be a set of generators of $J$, let $U$ be as in Proposition \ref{p:integral}, and let $x_0\in U$ such that $f(x_0)=0$. Then, for any real analytic arc $\mu:(\bR,0)\to(\bR^n,x_0)$, one has $\ord_{s}f(\mu(s))\geq\min_j\{\ord_{s}g_j(\mu(s))\}.$

In particular,  the equality of zero sets $\cZ(J)=\cZ(\overline J)$ holds.
\fin
\end{corollary}

Coming back to our deformation $F(x,t)$, we shall work with the presentation \eqref{eq:present} of \S\ref{ss:thm}, that we recall here: 
 
\begin{equation}\label{eq:pres}
  F(x,t) =f_0(x)+\sum_{j\geq 1}^\infty t^jf_j(x).
\end{equation}

Let us consider the following condition:

\begin{equation}\label{eq:cond2}
\left\{
\begin{aligned}
\quad &  \mbox{For any } x_{0}\in \Sing F_{0} \cap  \biggl\{\frac{\partial F}{\partial t}(x,0) =0\biggr\}, \mbox{ there is }  c_{x_{0}} >0  \mbox{ such that }  \\
 \quad  &   \biggl| \frac{\partial F}{\partial t} \biggr| \le  c_{x_{0}} \biggl\| \frac{\partial F}{\partial x_{1}}, \ldots , \frac{\partial F}{\partial x_{n}}  \biggr\|   \  \mbox{ when } (x,t) \to (x_{0},0).
\end{aligned}
\right.
\end{equation}

\begin{proposition}\label{p:delta2}
If condition \eqref{eq:cond2}  holds, then 
one has the equality of set germs $\Sing \widetilde F = \Sing F$ and,
in particular,  $\Delta = \{0\} \times \bK$. 

We have the following implications: 
\[ \eqref{eq:cond2}  \Longrightarrow  \eqref{eq:cond} \Longrightarrow \eqref{eq:cond0}.\]
\end{proposition}
\begin{proof}
Let us show the equality $\Sing F = \Sing  \widetilde F$.  The inclusion of set germs at the origin $\Sing F \subset  \Sing  \widetilde F$ is a direct  consequence of the definition of the singular loci.  The converse inclusion follows from condition \eqref{eq:cond2}. Indeed,  by \emph{reductio ad absurdum}, let us suppose
 that this does not hold. Using the Curve Selection Lemma, one then obtains an analytic path $\eta(s)=\bigl(x(s),t(s)\bigr)$  for $s\in[0,\epsilon)$, where $t(0)=0$ and $x(0)=0$, such that $\eta(s) \in \Sing  \widetilde F \m \Sing F$ for all $s\not= 0$. This means:  $\frac{\partial F}{\partial x_{i}}(\eta(s))= 0$  and $\frac{\partial F}{\partial t}(\eta(s))\not= 0$  for all $s\not= 0$ and all $i=1,\ldots, n$. This contradicts the inequality \eqref{eq:cond2} for the point $x_{0} =0$.  Here ends the proof of the inclusion $\Sing F \supset  \Sing  \widetilde F$, and thus of the equality $\Sing F =  \Sing  \widetilde F$. 

By restricting  this equality to the slice $t=0$, we get in particular the following inclusion  of the set germs at $0$:
\begin{equation}\label{eq:inclsets}
  \Sing F_{0}\subset \biggl\{ \frac{\partial F}{\partial t} (x,0) =0\biggr\}, 
\end{equation}
which shows the implication  \eqref{eq:cond2} $\implies$ \eqref{eq:cond}. 
 
\noindent The implication $\eqref{eq:cond} \Longrightarrow \eqref{eq:cond0}$ has been actually shown in 
the beginning of the proof of Theorem \ref{t:tame}, see \eqref{eq:inclsing}. 
Alternatively, let us remark that we may apply the  preceding argument  
at any fixed point $x_{0} \in \Sing F_{0}\m \{0\}$ to conclude that there is a neighbourhood $U(x_{0})$ within the inclusion \eqref{eq:inclsets} holds. This implies the inclusion \eqref{eq:inclsets} of set-germs at $0$, which concludes our proof. 
\end{proof}

\begin{theorem}\label{t:general}
 Let $F(x,t) = F_{t}(x)$ be a $\bK$-analytic deformation  of $F_{0}$  such that the Jacobian ideal
 $(\partial F_{t})$ is included in the integral closure $\overline{(\partial F_{0})}$ for all $t$ close enough to 0.
 
 Then the deformation $F$ of $F_{0}$ is a deformation with fibre constancy in the sense of Definition \ref{d:topconstantaway}, and $\widetilde F$ has a Milnor-Hamm tube fibration \eqref{eq:tube}, cf Definition \ref{d:tube}.
\end{theorem}


\begin{proof}
Let  $I=\overline{(\partial F_{0})}$ denote the integral closure of the Jacobian ideal $(\partial F_{0})$. 
By our hypothesis we have the inclusion $(\partial F_{t}) \subset  I$, for any $t$ close enough to 0.  
\begin{lemma}\label{l:gen}
The following inclusions hold:
\begin{enumerate}
 \rm \item \it  $(\partial f_{j}) \subset  I$, for any $j\in \bN$.
 \rm \item \it  $\bigl(\partial_{x} (\frac{\partial F}{\partial t})\bigr)\subset  I$.
\end{enumerate}
\end{lemma}

 
\begin{proof}
We will treat separately the two cases, $\bK = \bC$ and $\bK = \bR$.

\smallskip

\noindent \underline{Case $\bK = \bC$.} 
By  the invariance of the Jacobian ideal inclusion $(\partial F_{t}) \subset  I$, we have that 
$\sum_{j\geq 1}^\infty t^j\partial_i f_j(x)\in I$ for every $i\in\{1,\dots,n\}$, and for all 
 $t$ close enough to 0. Dividing out by  $t\neq 0$, we thus get:
 \begin{equation}\label{eq:ideal}
 \partial_i f_1+\sum_{j\geq 2}^\infty t^{j-1}\partial_i f_j(x)\in I.
 \end{equation}
We invoke the following classical theorem by Henri Cartan.  By the reason explained in Remark \ref{r:cartan} this holds only for $\bK = \bC$, and this is why we need to prove separately the real case of Theorem \ref{t:general}.
 
 \begin{theorem}\cite[p.194]{Ca} \label{t:cartan}\it \ \\
 Let $\cO_n$ be the ring of germs of holomorphic functions at $0\in\bC^n$ and
let $I$ be an ideal of $\cO_n$. Suppose that $V$ is a neighborhood of $0\in\bC^n$ and $g\in\cO(V)$. If there exists a sequence $(g_k)_{k\in \bN}$
of holomorphic functions $g_k\in \cO(V)$ such that $g_k\in I$ for all $k$ and $(g_k)_{k\in \bN}$ converges uniformly on compact subsets of $V$ to $g$, then $g\in I$.
 \end{theorem}

We choose a polydisk $\{(x,t)\in\bC^n\times\bC \mid |x_i|<\alpha_i,|t|<\delta,  i=1,\dots,n\}$ on which $F$ is holomorphic, i.e. for some well-chosen constants $\alpha_{i}>0$. Hence each $f_j$ is 
holomorphic on $V:=\{x\in\bC^n \mid |x_i|<\alpha_i,  i=1,\dots,n\}$.

 We choose some sequence $t_{k}\to 0$ and define $g_{k}(x) := \sum_{j\geq 2}^\infty t_{k}^{j-1}\partial_i f_j(x)$ which, by its definition, converges to the function germ 0 as $k\to \infty$, uniformly on compact sets.  Then
 Theorem \ref{t:cartan} applies,  and from \eqref{eq:ideal} we may deduce that $\partial_i f_1\in I$.
 
We may then apply inductively the same reasoning and obtain that  $\partial_i f_j\in I$ for every $j\ge 1$.  This concludes the proof of Lemma \ref{l:gen}(a) over $\bC$.

\smallskip

\noindent (b). From (a) it  follows that the infinite series  in variable $t$:
$$\partial_{i} \biggl(\frac{\partial F}{\partial t}\biggr) = \sum_{j=1}^\infty   j t^{j-1} \partial_{i}f_j(x)$$
  has all coefficients in $I$.
To show that this series converges to an element of $I$, we apply once again Cartan's Theorem \ref{t:cartan} to the finite partial sums $h_{p}(x,t) := \sum_{j=1}^p   j t^{j-1} \partial_{i}f_j(x)$, this time as function germs in variables $x$ and $t$.

This ends the proof of Lemma  \ref{l:gen} over  $\bC$.

\medskip

\noindent \underline{Case $\bK = \bR$.} 
We start from the presentation \eqref{eq:pres},  with real analytic function germs $f_{j}$, and where
 the Jacobian ideals $(\partial F_{0})$ and $(\partial F_{t})$  are  with real analytic coefficients. By $I=\overline{(\partial F_{0})}$ we will denote here the real integral closure.

We choose $\delta>0$, $\alpha_i>0$, $i\in\{1,\dots,n\}$,
such that the Taylor series of $F$ at $(0,0)\in\bR^n\times\bR$ converges on 
$\{|x_i|<\alpha_i \mid i=1,\dots,n\}\times\{|t|<\delta\}$. In particular the Taylor series of each $f_j$ converges on 
$\{|x_i|<\alpha_i \mid i=1,\dots,n\}\subset \bR^n$.

Let $F^c$ be the complexification of $F$ (i.e. if $\sum_{J\subset \bN^n, j\in\bN}a_{j,J}t^jx^J$
is the Taylor series of $F$ then $F^c(z,\lambda)=\sum_{J\subset \bN^n, j\in\bN}a_{j,J}\lambda^jz^J$ for $z_j,\lambda\in\bC$)
and let $f^c_j$  be the complexification of $f_j$.

We consider now the polydisk  $P\subset\bC^n\times\bC$, $P=\{(z,\lambda) \mid |z_i|<\alpha_i,|\lambda|<\delta\}$. Then 
$F^c$ is holomorphic on $P$.  Let $I_{\bC}$ be the ideal in $\cO_n$ generated by $\partial_zF^c_0$.

If $t\in\bR$, $|t|<\delta$, is such that  $\langle\partial_xF_t\rangle\subset I$, then it follows that 
 $\langle\partial_zF^c_t\rangle\subset I_{\bC}$. 
 We notice now that in the proof of the $\bC$-analytic case, we did not actually need the inclusion of
 the Jacobian ideals for all complex values of the parameter, namely some sequence convergent to $0$ suffices. We obtain that 
 $\partial f_j^c\subset I_{\bC}$.
  This means that $\partial_i f^c_j$ is a linear combination with complex analytic coefficients of the generators of  $I_{\bC}$.
  But since the restriction to $\bR^n$ of $\frac{\partial f^c_j}{\partial z_i}$ is $\frac{\partial f_j}{\partial x_i}$, hence a real function, it follows that, by taking the real part of each coefficient, we obtain a real linear combination equal 
  to $\partial_i f_j$. This shows that we actually have $\partial_i f_j\in I$, and the proof of point (a) is done.

\smallskip
 
 To show (b) we proceed in the same manner: we interpret the real analytic functions as complex analytic and we apply Cartan's theorem, as explained above, to get $\partial_{i} \bigl(\frac{\partial F}{\partial t}\bigr)\in I_{\bC}$. Now, by the same 
 arguments for the linear combination, we get that actually $\partial_{i} \bigl(\frac{\partial F}{\partial t}\bigr)\in I$.
\end{proof}

 \begin{remark}\label{r:cartan}
 Theorem \ref{t:cartan} is not true over $\bR$. A counter-example can be found for instance in \cite[Ch. 4 \S 6]{ABF}.  Namely, one constructs a sequence of real functions $\{f_k\}$ which are analytic on the real line $\bR$,  that converge 
uniformly on compacts to some analytic function $f$, and such that the germ at the origin of each $f_k$ is in some ideal $I$, but
that $f$ is not in $I$. The problem is that the radius  of  convergency of the Taylor series of $f_k$ at the origin goes to
zero as $k\to \infty$.  In contrast, note that the setting of our above proof insures  a stronger convergence, as defined e.g. in \cite[Definition 6.2]{ABF}.
 \end{remark}

\medskip

\noindent We continue the proof of Theorem \ref{t:general}. 
\medskip

Let us show that condition \eqref{eq:cond2} is satisfied.  We apply Proposition \ref{p:integral} to the ideal $(\partial F_0)$ and its generators $\partial_1F_0,\dots ,\partial_nF_0$, and for the 
choice of a neighbourhood $U$ of the origin such that the conclusion of Proposition \ref{p:integral} holds.

So let us fix some point $(x_{0},0)$ with $x_{0}\in  \Sing f_{0} \cap \Bigl\{\frac{\partial F}{\partial t} =0\Bigr\}_{| t=0}$. 

Let  us consider  an analytic path $\eta(s)=\bigl(x(s),t(s)\bigr)$  for $s\in[0,\epsilon)$, where $t(0)=0$ and 
 $x(0) = x_{0}$, and let us
  compute the limit, when $s\to 0$, of the fraction:
    \begin{equation}\label{eq:fracs}
\frac{\frac{\partial F}{\partial t}(x(s), t(s))}{\| \partial_{x} F (x(s), t(s) )\|}.
\end{equation}

By our hypothesis on the analytic path $\eta(s)$, we get $\lim_{s\to 0} \partial_jF_{t(s)}(x(s))  = \partial_jF_{0}(x(0)) = \partial_j f_{0}(x(0)) = 0$ for all $j=1, \ldots , n$.
It follows in particular that the limit of the fraction \eqref{eq:fracs}, when $s\to 0$, is of type ``$\frac{0}{0}$''.

\

Let then set $\kappa := \min_{i=1}^{n} \bigl\{ \ord_{s}\partial_if_0 (x(s)) \bigr\}$, and note that $\kappa >0$ since  $x(0)\in \Sing f_{0}$.  

\begin{lemma}\label{l:min}
 Let  $\ell\in \{1,\dots,n\}$ such that $\ord_{s}\partial_\ell f_0 (x(s))=\kappa$.
Then
 $$\ord_{s}\frac{\partial F}{\partial x_\ell}\bigl(x(s),t(s)\bigr)=\kappa.$$
\end{lemma}
\begin{proof}
By Lemma \ref{l:gen}, for every $j\in\bN$ we have  
$\partial_\ell f_j\in (\partial_x f_0) =I$. This implies that $\ord_{s}\partial_\ell f_j(x(s))\geq \kappa$. 
Since $t(0)=0$ we also have $\ord_s t^j(s)\geq 1$ for $j\geq 1$. We deduce that 
$\ord_s\sum_{j=1}^N t^j(s)\partial_\ell f_j(x(s))>\kappa$
and, since $\ord_{s}\partial_\ell f_0 (x(s))=\kappa$,
our claim follows.
\end{proof}

Our next claim is:
\begin{equation}\label{eq:kappa}
 \ord_s\frac{\partial F}{\partial t}(x(s),t(s))>\kappa. 
\end{equation}
To show this, we need the following basic result:  


\begin{lemma}\label{l:minord}
Let $x:[0,\varepsilon)\to\bK^n$ be an analytic path. 
Let  $h : (\bK^n, x(0)) \to (\bK, 0)$ be an analytic function germ such that $\partial_i h (x(0)) =0$ for all $i=1, \ldots , n$. Then: 
$$\ord_{s} h(x(s))>\min_{i}\ord_{s} \partial_i h (x(s)).$$
\end{lemma}
\begin{proof} 
As $h(0) =0$, we have:
 $$\ord_{s} h(x(s))> \ord_{s}\frac{d}{ds}\Big(h(x(s))\Big)=\sum_i \partial_i h (x(s)) x'(s)\geq \min_{i}\ord_{s} \partial_i h (x(s)).$$
\end{proof}

By Lemma \ref{l:gen}(a)  and Corollary \ref{c:integral} we have, for any $j\ge  0$, the inclusions of zero sets:
$$\cZ(f_{j}) \supset \cZ(\partial f_{j}) \supset \cZ(I) =\cZ(\partial f_{0}).$$
Therefore $h=f_{j}$ satisfies the hypotheses of  Lemma \ref{l:minord}, for any $j\ge 0$. Applying thus Lemma \ref{l:minord}, 
we deduce:
\begin{equation}\label{eq:minord}
\ord_{s} f_j(x(s))>\min_{i}\ord_{s} \partial_i f_j (x(s)).
\end{equation}

 Next, by Lemma \ref{l:gen} again, we have  $\partial_i f_j\in \overline{(\partial f_0)}$ for any $i\in\{1,\dots,n\}$, hence we get $\ord_s \partial_i f_j(x(s))\geq \kappa$ by Corollary \ref{c:integral} and our choice of $U$ satisfying Proposition \ref{p:integral}. Combining this with the inequality \eqref{eq:minord} we then get: 
 $$\ord_{s} f_j(x(s))>\kappa $$
 for  any $j\ge 0$, which shows that our claimed inequality \eqref{eq:kappa}  holds indeed.

\

Finally, by \eqref{eq:kappa} we get that the order of the numerator in \eqref{eq:fracs} is $>\kappa$. By Lemma \ref{l:min}
we get that the order of the denominator in \eqref{eq:fracs} is $=\kappa$. This implies that the limit of \eqref{eq:fracs} is 0, and therefore condition \eqref{eq:cond2} is satisfied.

By Proposition \ref{p:delta2} it then follows that condition \eqref{eq:cond} is satisfied, and thus by Theorem \ref{t:tame} and Theorem \ref{t:jt},  it  follows that  $F$ is tame,  that  $\widetilde F$ has a Milnor-Hamm tube fibration \eqref{eq:tube},  and  that  $F$ is a deformation with fibre constancy (in the sense of Definition \ref{d:topconstantaway}).

This finishes the proof of Theorem \ref{t:general}. 
 \end{proof}

\section{Examples}\label{examples}

By the following four examples we show here the variety of  situations that may occur: the Jacobian criterion of Theorem \ref{t:general} is satisfied in  \ref{ex:integralcl} but not in   \ref{ex:notJ}; the condition \eqref{eq:cond} is satisfied in \ref{ex:notJ} but not in \ref{ex:tame}; and finally,  the tameness condition holds in  \ref{ex:tame} but not in \ref{ex:nottame}.  We remind here too that, in all our notations, the maps and the sets defined by them are regarded as germs at the origin.

\begin{example}\label{ex:notJ}  
Let $F: (\bK^{2}\times \bK, 0) \to (\bK,0)$, $F (x,y,t) = y^{2}(x^{2}-(y-t)^{2})$.   This is a \emph{deformation of a line singularity}, in the terminology used  by Siersma in \cite{Si1}  in the complex setting only. 
For any parameter $t$ close enough to 0,   the function germ $F_{t}$ has indeed as singular locus $\Sing F_t := \{ y=0\}\subset \bK^{2}$, hence the same line.  

In our setting, we have the following set germs in $(\bK^{3},0)$:  $\Sing F =  \{x=0, y=t\} \cup \{ y=0\}$, and 
 $\Sing \widetilde F = \{x=0, y=t\} \cup \{x=0, 2y=t\} \cup \{y=0\}$.

Let us notice the strict inclusions:  $\bigcup_{t}\Sing F_{t} =  \{ y=0\} \subsetneq \Sing F \subsetneq  \Sing \widetilde F$.

The equality $ (\Sing F)_{| t=0} = (\Sing \widetilde F)_{| t=0}$ however holds, thus the preliminary condition \eqref{eq:cond0} is satisfied. The discriminant $\Delta := \widetilde F(\Sing \widetilde F) $ is the union of the germ  at 0 of the axis $\{0\} \times \bK$ with the germ of the curve in $\bK^{2}$ parametrized as $\{(-t^{4}/16, t)\}$.

By simple computations we see that condition \eqref{eq:cond} is satisfied, thus Theorem \ref{t:genP} holds, whereas condition \eqref{eq:cond2} is not satisfied precisely at the origin, and therefore the hypothesis of the  Jacobian criterion Theorem \ref{t:general} does not hold either.  In the complex setting,  such a deformation of a line singularity with singular locus $L$ was called ``admissible'' in \cite{Si1}, \cite{Si-icis} and in more other papers, and has the property that the line $L= \Sing F_{t}$  preserves its generic transversal type (which is $A_{1}$ in the above example),  while from the origin may spring, when $t\not= 0$,  finitely many special points along $L$ (i.e. with a different transversal type), as well as finitely many singular points outside $L$.  

In the complex setting, it turns out by a proof similar to L\^{e}-Saito's in \cite{LS} applied at some generic point $p\in L\m \{0\}$,  that the $t$-constancy of the generic transversal Milnor number implies condition \eqref{eq:cond}, thus the ``admissible'' deformations are in fact tame,  by our Theorem \ref{t:tame}.  
\end{example}

\begin{example}\label{ex:integralcl}
Let us consider the deformation $F(x,t)=z_1^5+z_2^5+z_1^6z_2^6z_3^2+tz_1^3z_2^3$ of the polynomial 
$F_0=z_1^5+z_2^5+z_1^6z_2^6z_3^2$ with $\Sing F_0 = \{ z_{1}= z_{2}= 0\}$, either over $\bC$ or over $\bR$. 
  Note that the inclusion $(\partial F_{t}) \subset (\partial F_0)$ doesn't hold since 
$z_1^3z_2^2\in \overline{(\partial F_0)}\setminus{(\partial F_0)}$
and $z_1^2z_3^2\in \overline{(\partial F_0)}\setminus{(\partial F_0)}$.

One easily checks that the inclusion $(\partial F_{t}) \subset \overline{(\partial F_0)}$ holds, hence our Theorem \ref{t:general} applies.
\end{example}

\begin{example}\label{ex:tame}
Starting from the Whitney umbrella equation $x^{2}+ y^{2}z=0$, let us consider the real analytic function $F_{0}(x,y,z) = x^{3}+ xy^{2}z$ and its deformation $F(x,y,z;t) = (x^{2}+ y^{2}z)(x-t)$.  

The singular locus $\Sing F_{0}$ consists of the germs of the two axes $\{x=y=0\}\cup\{x=z=0\}$, the first of multiplicity 4 and the second of multiplicity 1.
 For $t\not=0$,  the singular locus $\Sing F_{t}$ is the germ at $0\in \bR^{3}$ of $\{x=y=0\}$ with multiplicity 1. There are 
 two more curves in $\bR^{3}\times \{t\}$ which emerge from $\Sing F_{0}$:
$$ \bigl\{x= \frac23 t, \  y=0\bigr\}\cup \bigl\{x=t, \ x^{2} + y^{2}z =0\bigr\}$$ 
that are no more germs at $0\in \bR^{3}$. Their union over $t$, together with $\cup_{t}\Sing F_{t}$, constitute the singular set germ $\Sing \widetilde F$.

By straightforward computations, one establishes that  $F$ satisfies condition \eqref{eq:cond0},  that $F$ does not satisfy condition \eqref{eq:cond}, but that $F$ is tame, cf Definition \ref{d:tamedefo}.    Indeed, the Milnor set:
$$M(\widetilde F) = \{ y^{2}=2z^{2}, 3x^{2} - 2tx +4y^{5} =0\} \cup \{ y=z=0\}$$
 intersects $\{t=0\} \cap \Sing F_{0}$ at the origin of $\bR^{3}\times \bR$ only.

Our Theorem \ref{t:jt} then tells that this is a deformation with fibre constancy, and moreover that $\widetilde F$ has a Milnor-Hamm tube fibration. 

Let us point out that in \cite[Example 9.2]{Ho} the author considers the same $F_{0}$ in the complex analytic setting (moreover within a 2-parameter deformation) and  shows that his algebraic criterion \cite[Theorem 1.2]{Ho} applies to conclude that this is a deformation is ``admissible" in the terminology of  \cite{ST-mildefo}. Through this deformation, Hof can then compute  the homology of the Milnor fibre of $F_{0}$ by using the special  method developed \cite{ST-mildefo} in the complex setting. 
\end{example}
\begin{example}\label{ex:nottame}
Let us take a look at the  deformation $F=y^2+x^2(tz-x)$ discussed in \cite[Example 9.3]{Ho}. 

We have $\partial F=( -3x^2+2txz, 2y,  tx^2 ; \ x^{2}z )$, thus $\Sing F = \Sing \widetilde F$, and in particular our condition \eqref{eq:cond0} is satisfied.
Note that the Jacobian ideal $\partial F_{t}$ is not constant for $t$ in some neighbourhood of 0, despite the fact that the set germ $\Sing F_t$ is constant.

 The computation of the Milnor set over $\bR$ is easier here; we get that $M(\widetilde F)$  contains a surface germ in the  3-space $\{ y=0\} \subset \bR^{3}\times \bR$ of coordinates $x,z,t$, and the intersection of $M(\widetilde F)$ with the slice $t=0$ contains $\Sing F_0 = \{ (0,0, z,0) \}$.
This tells that $\widetilde F$ is not $\rho$-regular, which implies that the $\partial$-Thom regularity fails\footnote{There are examples of maps which fail to be Thom regular but they are however $\rho$-regular, cf \cite{ACT}. It remains here the question if this can happen in case of deformations.} along the $z$-axis.
As we have remarked at the end of the proof of  Theorem \ref{t:jt},  the $\rho$-regularity is equivalent to the fact that the tube map \eqref{eq:tube} is a proper stratified submersion.  

 Since this example over $\bR$ is not $\rho$-regular (i.e. not tame, cf Definition \ref{d:tamedefo}), it cannot satisfy any other criteria implying deformation with fibre constancy.    In the complex setting,  Hof  remarks in \cite{Ho} that this example does not satisfy his criteria; then he computes explicitly the fibres of the fibration \eqref{eq:faraway} and finds that  the homotopy type varies (actually one has $S^2$ for $t\not=0$ and $S^1\vee S^1$ for $t=0$), confirming that the deformation does not have fibre constancy. 
\end{example}



\vspace{\fill}

\end{document}